\documentclass[sts,reqno]{imsart}

\usepackage{amsthm,amssymb,amsmath}
\usepackage{color}
\usepackage{tikz}
\usetikzlibrary{arrows}

\usepackage[utf8]{inputenc}
\usepackage[shortlabels]{enumitem}
\usepackage{url}
\usepackage{etex}
\usepackage{alltt}
\usepackage{alg}
\usepackage{memhfixc}
\usepackage{bbold}
\usepackage{cases}
\usepackage[noadjust]{cite}
\definecolor{labelkey}{rgb}{0.0, 0.8, 0.3}
\usepackage[top=2.5cm,bottom=2.5cm,left=2.5cm,right=2.5cm,includeheadfoot,asymmetric]{geometry}

\numberwithin{equation}{section}
\newtheorem{thm}{Theorem}
\newtheorem*{thm*}{Theorem}
\newtheorem{rem}[thm]{Remark}
\newtheorem{defi}[thm]{Definition}

\newtheorem{cor}[thm]{Corollary}

\newcommand{\Sp}{S} 		%the metric or geodesic space all along
\newcommand{\prob}{\mathbb{P}}	%the probability measure
\renewcommand{\P}{P}	%again
\newcommand{\Q}{Q}	%alternative probability measure
\newcommand{\esp}{\mathbb{E}}
\newcommand{\E}{\mathbb{E}}
\newcommand{\bcdot}{\,\begin{picture}(-1,1)(-1,-3)\circle*{2}\end{picture}\ \,}

\newcommand{\bs}{b^\star}	%main barycenter
\newcommand{\mus}{\mu^\star} %barycenter in Wasserstein
\newcommand{\PS}{\mathcal{P}_2(\Sp)}		%Wasserstein space

\newcommand{\cP}{\mathcal{P}}
\newcommand{\R}{\mathbb R}

\newcommand{\ra}{\rightarrow}

\newcommand{\argmin}{\mathop{\mathrm{argmin}}}

\newcommand{\lin}{\mathcal L}
\newcommand{\curv}{\mathrm{curv}}

\renewcommand{\phi}{\varphi}
\newcommand{\ud}{\mathrm{d}}

\DeclareMathOperator{\supp}{supp}

\DeclareMathOperator{\id}{id}

\newcommand{\email}[1]{(\href{mailto:#1}{\rm{#1}})}

\renewcommand{\abstractname}{Abstract}

\usepackage{hyperref}
\hypersetup{
  colorlinks = true,
  urlcolor = blue, 
  linkcolor = blue,
  citecolor = blue}

\begin{document}

\begin{frontmatter}

	\title{Fast convergence of empirical barycenters in Alexandrov spaces and the Wasserstein space}
	\runtitle{Fast convergence of empirical barycenters}

	\author{ \fnms{Thibaut} \snm{Le Gouic}\thanksref{}\ead[label=tlg]{thibaut.le\_gouic@math.cnrs.fr},
		\fnms{Quentin} \snm{Paris}\ead[label=qp]{qparis@hse.ru},
		\fnms{Philippe} \snm{Rigollet}\ead[label=rigollet]{rigollet@math.mit.edu}
			~and~
		\fnms{Austin J.} \snm{Stromme}\ead[label=ajs]{astromme@mit.edu}
	}

	\address{{Thibaut Le Gouic}\\
		{National Research University} \\
		{Higher School of Economics (HSE)}\\
		{Faculty of Computer Science}\\
		{Moscow, Russia}\\
	    {Aix Marseille Univ, CNRS, Centrale Marseille, I2M}\\
	    {Marseille, France}\\
	    {\email{thibaut.le\_gouic@math.cnrs.fr}}
%		\printead{tlg}
	}
	
		\address{{Quentin Paris}\\
		{National Research University} \\
		{Higher School of Economics (HSE)}\\
		{Faculty of Computer Science}\\
		{Moscow, Russia}\\
		{\email{qparis@hse.ru}}
	}

	\address{{Philippe Rigollet}\\
		{Department of Mathematics} \\
		{Massachusetts Institute of Technology}\\
		{77 Massachusetts Avenue,}\\
		{Cambridge, MA 02139-4307, USA}\\
		{\email{rigollet@math.mit.edu}}
	}
	\address{{Austin Stromme}\\
		{Department of EECS} \\
		{Massachusetts Institute of Technology}\\
		{77 Massachusetts Avenue,}\\
		{Cambridge, MA 02139-4307, USA}\\
		{\email{astromme@mit.edu}}
	}

\runauthor{Le Gouic, Paris, Rigollet \& Stromme}

\begin{abstract}
This work establishes fast, i.e., parametric, rates of convergence for empirical barycenters over a large class of geodesic spaces with curvature bounds in the sense of Alexandrov. More specifically, we show that parametric rates of convergence are achievable under natural conditions that characterize the bi-extendibility of geodesics emanating from a barycenter.  These results largely advance the state-of-the-art on the subject both in terms of rates of convergence and the variety of spaces covered. In particular, our results apply to infinite-dimensional spaces such as the 2-Wasserstein space, where bi-extendibility of geodesics translates into regularity of Kantorovich potentials.
\end{abstract}

\end{frontmatter}

\section{Introduction}

The notion of \emph{average} is paramount across all of statistics. It is fundamental not only in theory but also in practice as it arises in nearly every known estimation method: the method of moments, empirical risk minimization, and more.
The law of large numbers, sometimes dubbed the fundamental law of statistics, ensures that the average of independent and identically distributed random variables converges to their common expected value.
While such results guarantee asymptotic validity of the average, modern statistics and machine learning focuses on non-asymptotic performance guarantees that hold for every sample size.
The goal of this paper is to provide such non-asymptotic guarantees for a generalized notion of average over metric spaces with the Wasserstein space as a prime example.

To illustrate our goal more precisely, consider the following example. Let $X, X_1, \ldots, X_n$ be i.i.d. (independent and identically distributed) copies of $X$  in a vector space equipped with a distance $d$. Denote by $\E[X]$ and $\bar X$ the expectation of $X$ and the average of the $X_i$'s respectively. In many instances, $\bar X$ converges to $\E[X]$ at a \emph{dimension-independent and parametric rate}.
Such convergence is characterized by the following inequality:
\begin{equation}
    \label{eq:variance_bound1}
    \E [d^2(\bar X,\E[X])]\le  \frac{c\sigma^2}{n}\,,
\end{equation}
where $c>0$ and $\sigma^2=\E[ d^2(X,\E[X])]$ denotes the variance of $X$. When the underlying space is a Hilbert space, and $d$ is the natural Hilbert metric, this inequality is in fact an equality with $c=1$. More generally, the bound~\eqref{eq:variance_bound1} holds in a Banach space if and only if it is of type~2~\cite{LedTal91}, which is  a property linked to the geometry of the Banach space~\cite{Kwa72}.

Using a predominantly geometric approach, we tackle the question of extending these statistical guarantees to metric spaces that are devoid of a linear structure. In recent years, this problem has gone from a theoretical curiosity to a practical necessity. Indeed, data such as images, shapes, networks, point clouds and even distributions are now routinely collected and come with the need for new statistical methods with strong performance guarantees.

\subsection{Barycenters}
The barycenter is a natural extension of the notion of expectation on an abstract metric space which we recall next. Let $(S,d)$ be a metric space and $\PS$ be the set of (Borel) probability measures $P$ on $S$ such that, the \emph{variance functional} defined by \[b \mapsto \int d^2(b,x)\ud P(x)\] 
is finite on $S$. 
For $P\in\PS$, a \emph{barycenter} of $P$ is any $\bs \in \Sp$ such that  
\[\bs\in\argmin_{b\in\Sp}\int d^2(b,x) \ud P(x).\]
Throughout, we always implicitly assume the existence of at least one barycenter. The question of existence has been largely addressed in the literature and shown to hold in most reasonable scenarios~\cite{Afs11,Le-Lou17}. 

Letting $X_1, \ldots, X_n$ be $n$ random variables drawn independently from $P$, an \emph{empirical barycenter}, is defined as a barycenter of the empirical distribution $P_n=(1/n)\sum_{i=1}^n\delta_{X_i}$, i.e.,
\[
b_n \in \argmin_{b \in S} \frac{1}{n}\sum_{i=1}^n d^2(b,X_i)\,.
\]
The main motivation for this work is to study the statistical behavior of empirical barycenters in the context where $(S,d)$ is the $2$-Wasserstein space  of probability measures over $\R^D$ that is defined by $S=\mathcal P_2(\R^D)$ equipped with the $2$-Wasserstein metric.

 The Wasserstein space has recently played a central role in many applications including machine learning~\cite{ArjChiBot17,GenPeyCut18}, computer graphics~\cite{SolGoePey15,FeyChaVia17}, statistics~\cite{RigWee18, PanZem19} and computational biology~\cite{SchShuTab19}.

Barycenters in this setting are called Wasserstein barycenters and were first studied in~\cite{AguCar11} who derive a dual formulation of the associated optimization problem and establish regularity properties of the solution. Despite computational advances~\cite{CutDou14, StaClaSol17} and asymptotic consistency results~\cite{AguCar17,Le-Lou17, ZemPan19}, little is  known about the rates of convergence of barycenters on Wasserstein spaces, let alone general metric spaces.
Initial contributions have focused on the Wasserstein space over the real line which is isometric to a convex subset of a Hilbert space~\cite{AguCar17} or finite dimensional subspaces of the Wasserstein space~\cite{KroSpoSuv19, chewi2020gradient}. 
The study of rates of convergence of Wasserstein barycenters in the general case was triggered in~\cite{AhiGouPar18}, where the truly geometric nature of the question was first brought forward. However, the rates obtained in that paper suffer from the curse of dimensionality: the rate $n^{-1}$ in~\eqref{eq:variance_bound1} is replaced with $n^{-1/D}$, whenever the probability distributions are defined on $\R^D, D\ge 3$ and their techniques yield the suboptimal rate $n^{-2/3}$ in the flat case $D=1$.
 
In this work, we aim to improve these results to optimality. To that end, we adopt a general framework and establish dimension-independent
and parametric rates for empirical barycenters, of the form~\eqref{eq:variance_bound1}, for a large family of geodesic metric spaces with curvature bounds in the sense of Alexandrov. In particular these spaces need not be finite-dimensional. More specifically, our results apply to geodesic spaces with positive curvature under a compelling synthetic geometric condition: the bi-extendibility of all geodesics emanating from a barycenter. 

The $2$-Wasserstein space over $\R^D$ is itself a geodesic space with positive curvature~\cite{AmbGigSav05} and parametric rates for this case follow as a simple consequence of our general result. In addition, we show that the bi-extendibility condition translates into simple regularity conditions on the Kantorovich potentials.

We end by mentioning references related to our work.
General properties of barycenters in metric spaces with bounded curvature are studied in~\cite{Stu03,Oht12,Yok16} and~\cite{Yok17}. Specific examples of spaces with negative curvature are studied in~\cite{HotHucLe13}.
The asymptotic properties of empirical barycenters, in the case where $\Sp$ is a Riemannian manifold, are addressed in~\cite{BhaPat03,BhaPat05} and~\cite{KenLe11}. The statistical properties of empirical barycenters in abstract metric spaces have only been considered in~\cite{Sh18} (for the case of negatively curved spaces) and~\cite{AhiGouPar18}.

\subsection{Overview of main results}
\label{sec:overview}

Our results rely on the notion of \emph{extendible geodesics}. For $\lambda_{\rm in},\lambda_{\rm out}>0$ we say that a constant-speed geodesic $\gamma:[0,1] \to \Sp$ is $(\lambda_{\rm in}, \lambda_{\rm out})$-extendible if there exists a constant-speed geodesic $\gamma^{+} \colon [-\lambda_{\rm in},1+\lambda_{\rm out}]\to \Sp$ such that $\gamma$ is the restriction of $\gamma^{+}$ to $[0,1]$. We now state our main result and refer the reader to Section~\ref{sec:main} for omitted definitions. The \emph{variance} $\sigma^2$ of a distribution $P \in \cP_2(S)$ is given by
$$
\sigma^2=\inf_{b \in S}\int d^2(b, x) \ud P(x)=\int  d^2(\bs, x) \ud P(x)\,,
$$
where $\bs$ is a barycenter of $P$.

{Our main theorem is stated below informally and follows directly from  Theorems \ref{thm:master} and \ref{thm:extendgeod}.}
\begin{thm}[Main theorem, informal]
\label{thm:maininformal}
Let $(\Sp,d)$ be a geodesic space of nonnegative curvature. Let $\P\in\PS$ be such that, for any $x$ in the support of $\P$, there exist a geodesic from $\bs$ to $x$ that is $(\lambda_{\rm in},\lambda_{\rm out})$-extendible. Define
\[
k:= \frac{\lambda_{\rm out}}{1+\lambda_{\rm out}}-\frac{1}{\lambda_{\rm in}}\,.
\]
Then if $k>0$, $\bs$ is unique and for any  empirical barycenter $b_n$, we have
\[
\E[d^2(b_n,b^\star)]\le \frac{4\sigma^2}{k n}\,.
\]

\end{thm}

{Note that the constant $k$ is a geometric characteristic of the metric measure space $(S,d,P)$. Since the seminal work of Cheeger and Colding~\cite{CheCol96,CheCol97,CheCol00a,CheCol00}, and culminating with the Lott-Sturm-Villani theory of Ricci curvature lower bounds for geodesic spaces~\cite{Stu06,Stu06a, LotVil09}, such triples are central to modern geometry.}

To illustrate the meaning of the condition $k>0$, consider the canonical example of the Euclidean unit sphere (in any dimension). In this setup, the $(\lambda_{\rm in},\lambda_{\rm out})$-extendibility imposes that, for all $x$ in the support of $P$,
\[d(b^\star,x)(1+ \lambda_{\rm in}+\lambda_{\rm out})\le \pi
\]
In addition, condition $k>0$ requires that 
\[1+\lambda_{\rm in}+\lambda_{\rm out}\ge 2+\lambda_{\rm out}+\frac{1}{\lambda_{\rm out}}\ge 4,
\]
so that the support of $P$ is necessarily contained in the spherical ball of center $b^*$ and radius $\pi/4$. This bound aligns with known results, on the sphere, stating that we cannot hope for $1/n$ rates when mass is allowed to accumulate around a point antipodal to $b^\star$~\cite{EltHuc19}. The above theorem thus requires stringent conditions on the support of $P$ that can be largely relaxed using a local notion of extendibility where $\lambda_{\rm in}$ and $\lambda_{\rm out}$ may depend on $x$ and $k$ is only required to be positive on average (see Theorem~\ref{thm:master2}).

It is known that extendibility of geodesics  in the 2-Wasserstein space is related to the regularity of Kantorovich potentials~\cite{AhiGouPar18}. In this case, we get the following corollary.

\begin{cor}
Let $\P\in\cP_2(\cP_2(\R^D))$ be a probability measure on the 2-Wasserstein space and let $\mus\in\cP_2(\R^D)$ be a barycenter of $P$.  Let $\alpha,\beta>0$ and suppose that every $\mu\in\supp(P)$ is the pushforward of $\mu^{\star}$ by the gradient of an $\alpha$-strongly convex and $\beta$-smooth function $\varphi_{\mus\to\mu}$, i.e., $\mu=(\nabla\varphi_{\mus\to\mu})_{\#}\mus$. Then, if $\beta-\alpha<1$, $
\mus$ is unique and any empirical barycenter $\mu_n$ of $P$ satisfies
\[
\esp W_2^2(\mu_n,\mus)\le  \frac{4\sigma^2}{(1-\beta+\alpha)n}.
\]
\end{cor}

While our results are resolutely focused on spaces with nonnegative curvature, we obtain the following strong result about spaces of nonpositive curvature as a byproduct of the techniques developed in this paper.

\begin{thm} 
Suppose $(\Sp,d)$ is a geodesic space satisfying
$\kappa \le \curv(\Sp) \le 0$ for some $\kappa \le0$. Then, any distribution $\P\in\PS$ has a unique barycenter $\bs$ and any empirical barycenter $b_n$ satisfies
\[
\E d^2(b_n, \bs) \le \frac{\sigma^2}{n}.
\]
\end{thm}

The above parametric rate should be contrasted with the slower rates obtained in~\cite{AhiGouPar18} and~\cite{Sh18} but that hold without curvature lower bounds.

\section{Geometry and barycenters}
\label{sec:gb}
\subsection{Preliminaries}
\label{subsec:prelim}

We gather important notation and standard facts from metric geometry for reference. We refer the reader to~\cite{BurBurIva01} and~\cite{AleKapPet19} for an introduction.

Throughout, let $(S,d)$ be a (Polish) geodesic metric space (see Remark \ref{rem:mes}).
For all $x,y\in\Sp$, we call (constant-speed) geodesic connecting $x$ to $y$ any continuous path $\gamma:[a,b]\to\Sp$ such that $\gamma(a)=x$, $\gamma(b)=y$ and such that, for all $a\le s\le t\le b$, \[d(\gamma(s),\gamma(t))=\frac{t-s}{b-a}d(x,y).\] 
Given $\kappa\in \R$, we denote by $(M^2_{\kappa},d_{\kappa})$ the complete and simply connected $2$-dimensional Riemannian manifold with constant sectional curvature $\kappa$. The diameter $D_{\kappa}$ of $M^2_{\kappa}$ is $D_{\kappa} =+\infty$ if $\kappa\le 0$ and $D_{\kappa}=\pi/\sqrt{\kappa}$ if $\kappa>0$. We define a triangle in $\Sp$ as any set of three points $\{p,x,y\}\subset\Sp$. For $\kappa\in\R$, a comparison triangle for $\{p,x,y\}\subset S$ in $M^2_{\kappa}$ is an isometric embedding $\{\bar p,\bar x,\bar y\}\subset M^2_{\kappa}$ of $\{p,x,y\}$ in $M^2_{\kappa}$. Such a comparison triangle always exists and is unique (up to an isometry) provided $\mathrm{peri}\{p,x,y\}=d(p,x)+d(p,y)+d(x,y)< 2D_{\kappa}$. We recall the definition of curvature bounds for reference.
\begin{defi}
\label{def:curvk}
Let $\kappa\in \R$.  We say that ${\rm curv}(\Sp)\ge\kappa$ if for any triangle $\{p,x,y\}\subset\Sp$ with $\mathrm{peri}\{p,x,y\}< 2D_{\kappa}$, any comparison triangle $\{\bar p,\bar x,\bar y\}\subset M^2_{\kappa}$ and any geodesic $\gamma:[0,1]\to\Sp$ connecting $x$ to $y$ in $\Sp$, we have
\begin{equation}
\label{boundkappabelow}
d(p,\gamma(t))\ge d_{\kappa}( \bar p,\bar\gamma(t)),
\end{equation}
for all $t\in[0,1]$, where $\bar\gamma:[0,1]\to M^2_{\kappa}$ is the unique geodesic connecting $\bar x$ to $\bar y$ in $M^2_{\kappa}$. We say that $\curv(\Sp)\le \kappa$ if the same holds with converse inequality in \eqref{boundkappabelow}.
\end{defi}

Comparison angles allow to provide useful characterisations of curvature bounds.
Given $p,x,y\in\Sp$ with $p\notin\{x,y\}$ and $\textrm{peri}\{p,x,y\}<2D_{\kappa}$, we define the \emph{comparison angle} $\sphericalangle^{\kappa}_p(x,y)\in[0,\pi]$ at $p$ by 
 \[\cos\sphericalangle^{\kappa}_p(x,y):=\left\{
 \begin{array}{ll}
 \dfrac{d^2(p,x)+d^2(p,y)-d^2(x,y)}{2d(p,x)d(p,y)}&\mbox{ if }\kappa=0,\vspace{0.2cm}\\
 \dfrac{c_{\kappa}(d(x,y))-c_{\kappa}(d(p,x))\cdot c_{\kappa}(d(p,y))}{\kappa\cdot s_{\kappa}(d(p,x))s_{\kappa}(d(p,y))}&\mbox{ if }\kappa\ne 0,
 \end{array}
 \right.\]
  where, for $r\ge 0$,
 \begin{equation}
 \label{def:csk}
 s_{\kappa}(r):=\left\{
 \begin{array}{ll}
 \sin(r\sqrt\kappa)/\sqrt\kappa &\mbox{ if }\kappa>0,\\
 \sinh(r\sqrt{-\kappa})/\sqrt{-\kappa}&\mbox{ if }\kappa< 0,
 \end{array}
 \right.
 \end{equation}
and $c_{\kappa}(r)=s'_{\kappa}(r)$. 

{The next result can be found in~\cite{BurGroPel92}}.

\begin{thm}
\label{thm:curvequiv}
Let $(\Sp,d)$ be a geodesic space and $\kappa\in \R$. Then the following statements are equivalent.
\begin{enumerate}[label=\rm{(\arabic*)}]
    \item $\curv(\Sp)\ge \kappa$ in the sense of Definition \ref{def:curvk}.
    \item\label{it:quadruple} For all $p,x,y,z\in\Sp$ such that $p\notin\{x,y,z\}$,
\[\sphericalangle^{\kappa}_{p}(x,y)+\sphericalangle^{\kappa}_{p}(x,z)+\sphericalangle^{\kappa}_{p}(y,z)\le 2\pi.\]
 \item\label{it:anglemonotonicity} For any $p,x,y\in\Sp$ with $p\notin\{x,y\}$ any geodesics $\gamma_{x},\gamma_y:[0,1]\to\Sp$ from $p$ to $x$ and from $p$ to $y$ respectively, the function
 \[(s,t)\mapsto\sphericalangle^{\kappa}_p(\gamma_x(s),\gamma_y(t)),\]
 is non-increasing in each variable when the other is fixed.
\end{enumerate}
\end{thm}

Tangent cones are fundamental objects around which revolve many ideas of the paper. While they reduce to more familiar tangent spaces when $S$ is a smooth manifold, tangent cones provide a natural proxy for tangent spaces in the context of abstract metric structures.  Given $p\in \Sp$, let $\Gamma_p$ be the set of all geodesics $\gamma:[0,1]\to S$ with $\gamma(0)=p$. For $\gamma,\sigma\in\Gamma_p$, the \emph{Alexandrov angle} $\sphericalangle_p(\gamma,\sigma)$ is given by
\[\sphericalangle_p(\gamma,\sigma)= \lim_{s,t\to 0}\sphericalangle^{0}_p(\gamma(s),\sigma(t)).\]
Whenever $\mathrm{curv}(\Sp)\ge \kappa$ or $\mathrm{curv}(\Sp)\le \kappa$, for $\kappa\in\R$, the angle $\sphericalangle_p:\Gamma^2_p\to [0,\pi]$ exists and is a pseudo-metric on $\Gamma_p$. Hence, it induces a metric, still denoted $\sphericalangle_p$, on the quotient space $\Gamma_p/\sim$ where $\gamma\sim\sigma$ if and only if $\sphericalangle_p(\gamma,\sigma)=0$.
The completion $(\Sigma_p,\sphericalangle_p)$ of $(\Gamma_p/\sim,\sphericalangle_p)$ is called the \emph{space of directions} at $p$.
The \emph{tangent cone} $T_p\Sp$ of $(\Sp,d)$ at $p$ is the Euclidean cone over the space of directions $\Sigma_p$, i.e., the quotient space $\Sigma_p\times\R^+/\approx$ where $(\gamma,s)\approx(\sigma,t)$ if and only if $s=t=0$ or $(\gamma,s)=(\sigma,t)$.
For $(\gamma,s)\in \Sigma_p\times\R^+$, we denote $[\gamma,s]$ the corresponding element in $T_pS$ and we call $o_p=[\gamma,0]$ the tip of the cone.
For all $u=[\gamma,s],v=[\sigma,t]\in T_p S$, we set\footnote{Observe that we often denote paths and directions by the same letters.
For any two directions $\gamma,\sigma\in \Sigma_p$, the angle $\sphericalangle_p(\gamma,\sigma)$ is defined without ambiguity as the Alexandrov angle between any two representatives of these directions.}
\[\|u-v\|^2_p=s^2+t^2-2st\cos\sphericalangle_p(\gamma,\sigma).\]
We also use notation $\|u\|_p=\|u-o_p\|_p$ and $\langle u,v\rangle_p= st\cos \sphericalangle_p(\gamma,\sigma)$ so that 
\begin{equation}
\label{eq:polarize}    
\|u-v\|^2_p=\|u\|^2_p+\|v\|^2_p-2\langle u,v\rangle_p\,.
\end{equation}
The map $(u,v)\mapsto \|u-v\|_p$ is a metric on $T_p\Sp$ called the \emph{cone metric}. The cone structure of $T_p\Sp$ is defined, for $u=[\gamma,s]$ and $\lambda\ge 0$, by $\lambda\cdot u=[\gamma, \lambda s]$.

Let $C_p\subset\Sp$ denote the cut-locus of $p$ and for all $x\in\Sp\setminus C_p$, let $\uparrow^x_{p} \in \Sigma_p$ denote the direction 

of the unique geodesic connecting $p$ to $x$. Then, the \emph{log map} at $p$ is the map from $\Sp\setminus C_p$ to $ T_p\Sp$ that is defined by
 \[\log_p(x)= [\uparrow^{x}_p,d(p,x)]\,.\]
 
This definition can be extended to $x \in C_p$ by selecting an arbitrary direction from $p$ to $x$.

\begin{rem}
\label{rem:mes}
In the rest of this paper, we integrate functionals of the log map and now comment on its measurability.
Note first that the map $\log_p: S \to T_p S$ can be chosen to be measurable whenever the tangent cone $T_p\Sp$ is equipped with the $\sigma$-algebra $\mathfrak{S}_B$ generated by the closed balls.
This observation follows from a straightforward adaptation of Lemmas~3.3 and~4.2 in~\cite{Oht12} and implies Borel-measurability of the log map if the tangent cone $T_p\Sp$ is separable, since, in that case, $\mathfrak{S}_B$ coincides with the Borel $\sigma-$algebra.

In particular, when $\Sp=\mathcal P_2(X)$ is the $2$-Wasserstein space over a separable space $X$, this condition follows from the characterization of $T_p\mathcal P_2(X)$ in~\cite[Section 12.4]{AmbGigSav05}. 
It is unclear whether general conditions on $\Sp$ imply separability of the tangent cone---in fact, this question was raised in~\cite{Oht12} in the case where $\Sp$ is proper---or even Borel-measurability of the log map. Fortunately,
it can be checked that our arguments require Borel measurability not of $\log_p$ itself, but only of maps of the form $\langle \log_p(\cdot),b\rangle_p$ for some $b\in T_p\Sp$, which, in turn, follows from the $\mathfrak{S}_B$-measurability of $\log_p$.

\end{rem}
A useful consequence of the monotonicity property, Theorem \ref{thm:curvequiv} point \ref{it:anglemonotonicity}, is the following. If $\curv(\Sp)\ge 0$ (resp $\le 0$), then for any $x,y,p\in\Sp$ the metric $d$ and cone metric $\|\cdot-\cdot\|_{p}$ satisfy
\begin{equation}
    \label{eq:comparemetrics}
d(x,y)\le \|\log_p(x)-\log_p(y)\|_p\,,\quad\mbox{(resp. $\ge$)},
\end{equation}
with equality if $x=p$ or $y=p$.

\subsection{Tangent cone at a barycenter}

The notation $\|\cdot\|_p$ and $\langle \cdot,\cdot\rangle_p$ suggests that the cone $T_p\Sp$ possesses a Hilbert-like structure.
However,  the tangent cone may fail to be even geodesic whenever $S$ has only lower bounded curvature (see \cite{Hal00}).
To overcome this limitation, we gather facts on the structure of $T_pS$, that are used in the sequel, in particular when $p$ is a barycenter.

\begin{thm}
\label{thm:toolbox}
Let $\kappa\in \R$ and let $(\Sp,d)$ be a geodesic space with $\mathrm{curv}(\Sp)\ge \kappa$. Then the following holds:
\begin{enumerate}[label=\rm{(\arabic*)}]
    \item\label{it:tanlowbound} For any $p\in \Sp$, the tangent cone $T_p\Sp$ can be isometrically embedded in a geodesic space $\mathcal T_p\Sp$, such that $\mathrm{curv}(\mathcal T_p\Sp)\ge 0$.
    \item\label{it:baryisexpbary} For any $\P\in \PS$ and any barycenter $\bs\in\Sp$ of $P$,
    \[\iint\langle\log_{\bs}(x),\log_{\bs}(y)\rangle_{\bs}\ud P(x)\ud P(y)=0.\]
    \item\label{it:suppflat} For any $\P\in \PS$ and any barycenter $\bs$ of $\P$,  there exists a subset $\lin_{\bs}\Sp\subset T_{\bs}\Sp$, which is a Hilbert space when equipped with the restricted cone metric, and such that $\log_{\bs}(\supp(P))\subset \lin_{\bs}\Sp$.
    \item\label{it:espscal} For any $\P\in\PS$, any barycenter $\bs$ of $\P$, any $\Q\in\PS$ with $\log_{\bs}(\supp(Q))\subset \lin_{\bs}\Sp$ and any $b\in \Sp$, we have
    \[\int\langle\log_{\bs}(x),\log_{\bs}(b)\rangle_{\bs}\ud Q(x)=\langle u ,\log_{\bs}(b)\rangle_{\bs},\]
    where $u$ is the (Pettis) integral $\int v \ud \bar Q(v), \bar Q=(\log_{\bs})_{\#} Q$.
\end{enumerate}
\end{thm}

\begin{proof}
\ref{it:tanlowbound} This statement follows directly by combining Claim 3.3.2,  Theorem 3.4.1 and Theorem 11.3.1 in \cite{AleKapPet19}. Note that $\mathcal T_p\Sp$ is formally described as the ultralimit of blowups of $\Sp$ at point $p$ and called the \emph{ultratangent cone. }

\ref{it:baryisexpbary}  Let $X_1,\dots,X_n$ be i.i.d. with distribution $\P$ and let ${\P}_n=n^{-1}\sum_{i=1}^n\delta_{X_i}$. Let $f:\Sp^2 \to \R$ be defined by $f(x,y)=\langle\log_{\bs}(x),\log_{\bs}(y)\rangle_{\bs}$ and observe that $f \in L^2(\P\otimes\P)$. Moreover, explicit computations show that
 \[\esp|(\P\otimes\P)f-(\P_n\otimes \P_n)f|^2\underset{n\to+\infty}{\longrightarrow}0,\]
 so that 
 there exists a subsequence  $(\P_{n_k})$ such that $(\P_{n_k}\otimes \P_{n_k})f$ converges almost surely to $(\P\otimes \P)f$. Since $\P_{n_k}$ is finitely supported, it follows from~\cite[Proposition~3.2]{LanSch97} that $(\P_{n_k}\otimes \P_{n_k}) f\ge 0$. Therefore,  $(\P\otimes\P)f\ge 0$. Note that this first inequality holds even when $\bs$ is not a barycenter.
 
 To complete the proof, we prove the reverse inequality.

Without loss of generality, suppose that $\kappa<0$. Then, combining claim 1.1.1.d and Lemma 5.3.1 in \cite{AleKapPet19}, we deduce there exists a constant $c>0$ (depending only on $\kappa$) such that, for all $x,y\in \Sp$,
  \begin{equation}
  \label{eq:ineqangles}
  \sphericalangle^{\kappa}_{\bs}(x,y)\le \sphericalangle^0_{\bs}(x,y)\le \sphericalangle^{\kappa}_{\bs}(x,y)+ cd(\bs,x)d(\bs,y).
  \end{equation}
Now for all $x,y\in \Sp$, let $\gamma_x$ and $\gamma_y$ be two geodesics connecting $\bs$ to $x$ and $y$ respectively and with respective directions $\uparrow^x_{\bs}$ and $\uparrow^y_{\bs}$. Then, letting $\langle x,y\rangle^{0}_{\bs} = d(\bs,x)d(\bs,y)\cos\sphericalangle^{0}_{\bs}(x,y)$, we get
  \begin{align}
      \langle \log_{\bs}(x),\log_{\bs}(y)\rangle_{\bs}&=d(\bs,x)d(\bs,y)\cos\sphericalangle_{\bs}(\uparrow^x_{\bs},\uparrow^y_{\bs})
      \nonumber\\
      &=\lim_{s,t\to 0} d(\bs,x)d(\bs,y)\cos\sphericalangle^{0}_{\bs}(\gamma_x(s),\gamma_y(t))
      \nonumber\\
      &\le \lim_{s,t\to 0} d(\bs,x)d(\bs,y)\cos\sphericalangle^{\kappa}_{\bs}(\gamma_x(s),\gamma_y(t))
      \nonumber\\
      &\le d(\bs,x)d(\bs,y)\cos\sphericalangle^{\kappa}_{\bs}(x,y)
      \nonumber\\
      &\le \langle x,y\rangle^{0}_{\bs}+cd^2(\bs,x)d^2(\bs,y),
      \nonumber
  \end{align}
  according to Theorem~\ref{thm:curvequiv}-\ref{it:anglemonotonicity}, inequality~\eqref{eq:ineqangles} and the Lipschitz continuity of $\cos(\cdot)$.
  Now observe that, for all $x\in\Sp$ and all $t\in(0,1
  )$,
  \begin{align*}
      &d^2(x,\gamma_y(t))-d^2(\bs,x)\\
      &=d^2(\bs,\gamma_y(t))-2\langle x,y\rangle^{0}_{\bs}\\
      &\le d^2(\bs,\gamma_y(t))-2\langle \log_{\bs}(x),\log_{\bs}(\gamma_y(t))\rangle_{\bs}+cd^2(\bs,x)d^2(\bs,\gamma_y(t))\\
      &=t^2d^2(\bs,y)-2t\langle \log_{\bs}(x),\log_{\bs}(y)\rangle_{\bs}+ct^2d^2(\bs,x)d^2(\bs,y).
  \end{align*}
  Integrating with respect to $\ud P(x)$, using the definition of a barycenter and letting $t$ go to $0$, we deduce that 
  \[\int\langle\log_{\bs}(x),\log_{\bs}(y)\rangle_{\bs}\ud P(x)\le 0\,.\]
Since this holds for all $y\in \Sp$, integrating with respect to $\ud P(y)$ gives the result.

\ref{it:suppflat} Given $p\in\Sp$ and two elements $u=[\gamma,s],v=[\sigma,t]\in T_p\Sp$, we write $u+v=0$ and say that $u$ and $v$ are \emph{opposite} to each other if ($s=t=0$) or ($s=t$ and $\sphericalangle_{p}(\gamma,\sigma)=\pi$). Then, we set $$\lin_p\Sp=\{u\in T_p\Sp:\exists \,v\in T_p\Sp,u+v=0\}\,.
$$
Given $u\in \lin_p\Sp$, there is a unique vector opposite to $u$ denoted $-u$. The fact that $\lin_p\Sp$ is a Hilbert space for the restricted cone metric follows from Theorem 11.6.4 in \cite{AleKapPet19}. The inclusion $\log_{\bs}(\supp(P))\subset \lin_{\bs}\Sp$ follows from statement~\ref{it:baryisexpbary} and  \cite{Le-19}. 

\ref{it:espscal} Fix a barycenter $\bs$ of $\P$ and some $b\in \Sp$. Note that if $b$ is in the support of $P$, it follows readily from~\ref{it:suppflat} that~\ref{it:espscal} holds. The main purpose of this proof is to show that the same holds for any $b \in \Sp$, not necessarily in $\supp(\P)$. For brevity we use the following notation:
$$
|x-y|=\|\log_{\bs}(x)-\log_{\bs}(y)\|_{\bs}\,, \quad |x|=\|\log_{\bs}(x)\|_{\bs}
$$ 
and 
$$
\langle x,y\rangle=\langle\log_{\bs}(x),\log_{\bs}(y)\rangle_{\bs}\, \quad \forall \, x,y\in\Sp\,.
$$
In addition, whenever $u\in T_{\bs}\Sp$ and $y\in\Sp$, the notation $|u-y|$, $|u|$ and  $\langle u,y\rangle$ should be understood as $\|u-\log_{\bs}(y)\|_{\bs}$, $\|u\|_{\bs}$ and $\langle u,\log_{\bs}(y)\rangle_{\bs}$ respectively. 

We first prove that $u\in \lin_{\bs}\Sp \mapsto \langle u,b\rangle$ is a convex function. To that end, fix $t\in (0,1)$, $u_0,u_1\in\lin_{\bs}\Sp$, and set $u_t=(1-t)u_0+tu_1 \in\lin_{\bs}\Sp$. The path $t\in[0,1]\mapsto u_t$ defines a geodesic in $\mathcal T_{\bs}\Sp$ according to~\ref{it:tanlowbound} and~\ref{it:suppflat}. Since $\curv(\mathcal T_{\bs}\Sp)\ge 0$ we get from Definition~\ref{def:curvk} that, for all $y\in\Sp$ and  $t\in[0,1]$, 
\[|u_t-y|^2\ge (1-t)|u_0-y|^2+t|u_1-y|^2-t(1-t)|u_0-u_1|^2,\]
with equality if $y=\bs$ since $u_t,u_0$ and $u_1$ belong to $\lin_{\bs}\Sp$. Hence, by~\eqref{eq:polarize}, we have
\begin{align*}
2\langle u_t,b\rangle&=|u_t|^2+|b|^2-|u_t-b|^2\\
&\le (1-t)(|u_0|^2+|b|^2-|u_0-b|^2)+t(|u_1|^2+|b|^2-|u_1-b|^2)\\
&=(1-t)2\langle u_0,b\rangle+t2\langle u_1,b\rangle,
\end{align*}
so that $u \mapsto \langle u, b\rangle$ is a convex function on the Hilbert space $\lin_{\bs}\Sp$.
Therefore, for all $u\in\lin_{\bs}S$, there exists $g=g(u)$ such that for all $v\in\supp Q$,
\[
\langle \log_{\bs}(v),b\rangle\ge \langle u,b\rangle +  \langle \log_{\bs}(v)-u,g\rangle.
\]
Choosing $u$ to be the Pettis integral $\int \log_{\bs} (v) \ud Q(v)$ and integrating $v$ with respect to $Q$ yields
\begin{equation}
\label{eq:Jensen}
\int\langle \log_{\bs} (v),b\rangle\ud Q(v)\ge \langle u,b\rangle,
\end{equation}
The same arguments apply to the map $u\mapsto \langle-u,b\rangle$, so that
\begin{equation}
\label{eq:Jensen2}
\int\langle -\log_{\bs} (v),b\rangle\ud  Q(v)\ge \langle -u,b\rangle.
\end{equation}
To conclude the proof, it remains to show that $\langle -v,b\rangle=-\langle v,b\rangle$ for all $v\in\lin_{\bs}\Sp$ (in particular, for $u \in \lin_{\bs}\Sp$). To that end, select $v=[\gamma,s]\in\lin_{\bs}\Sp$ and let $-v=[\sigma,s]$. Then on the one hand, using the definition of opposite points and the triangle inequality on the space of directions $\Sigma_p$, it follows that 
$
\pi=\sphericalangle_{\bs}(\gamma,\sigma)\le\sphericalangle_{\bs}(\gamma,\uparrow_{\bs}^b)+ \sphericalangle_{\bs}(\uparrow_{\bs}^b,\sigma)$. On the other hand, the fact that $\mathrm {curv}(\mathcal T_p\Sp)\ge 0$ implies that 
\[\sphericalangle_{\bs}(\gamma,\sigma)+\sphericalangle_{\bs}(\gamma,\uparrow_{\bs}^b)+ \sphericalangle_{\bs}(\uparrow_{\bs}^b,\sigma)\le 2\pi,\]
according to Theorem~\ref{thm:curvequiv}-\ref{it:quadruple}. Combining these inequalities we get that $\sphericalangle_{\bs}(\gamma,\uparrow_{\bs}^b)+ \sphericalangle_{\bs}(\uparrow_{\bs}^b,\sigma)=\pi$ so that, by definition of the bracket $\langle .,.\rangle$ and using elementary trigonometry, we get $\langle -v,b\rangle=-\langle v,b\rangle$.
\end{proof}

\section{Main results}
\label{sec:main}

We now turn to the proof of our main result: parametric and dimension-free rates for empirical barycenters on Alexandrov spaces with curvature bounded below. We study the case of positively curved spaces in more detail and, in particular, the 2-Wasserstein space. 

\subsection{Hugging}

We begin with a useful identity that controls the average curvature of the square distance around its minimum.

Given $\P\in\PS$ and a barycenter $\bs$ of $P$, define for all $x, b \in S$ the \emph{hugging function} at $\bs$ by
\begin{equation}
    \label{eq:kb}
 k^b_{\bs}(x)=1-\frac{\|\log_{\bs}(x)-\log_{\bs}(b)\|^2_{\bs}-d^2(x,b)}{d^2(b,\bs)}.
\end{equation}
Note that if $\Sp$ is a Hilbert space, then $k^b_{\bs}(x)=1$ for all $x,b\in S$. The hugging function
$k^b_{\bs}$ measures the proximity of $\Sp$ to its tangent cone $T_{\bs}\Sp$.
The next result is a generalization of Theorem 3.2 in \cite{AhiGouPar18} and demonstrates the central role of the hugging coefficient in the context of barycenters: it precisely controls the quadratic growth of the variance functional around its minimum.

\begin{thm}[Variance equality]
\label{thm:varequal}
Let $\kappa\in\R$ and $(S,d)$ be a geodesic space with $\curv(\Sp)\ge \kappa$. Let $\P\in\PS$ and $\bs$ be a barycenter of $\P$. Then, for all $b\in \Sp$,
\begin{equation}
\label{eq:vi}
d^2(b,\bs) \int k^{b}_{\bs}(x)\,{\rm d}\P(x)= \int(d^2(x,b)-d^2(x,\bs))\,{\rm d}\P(x),
\end{equation}
where $k^b_{\bs}(x)$ is as in \eqref{eq:kb}. 
\end{thm}
\begin{proof}We adopt the same notation as in the proof of Theorem~\ref{thm:toolbox}-\ref{it:espscal}. By definition of the cone metric $|\cdot-\cdot|$ and bracket $\langle\cdot,\cdot\rangle$ in $T_{\bs}S$,  we get
\begin{align*}
    k_{\bs}^b(x) d^2(b,\bs) &= d^2(b, \bs) + d^2(x, b) - |b - x|^2 \\
    &= d^2(x, b) - d^2(x, \bs) + 2 \langle x, b \rangle.
\end{align*}
The result follows by integrating both sides with respect to $\ud P(x)$, in Theorem~\ref{thm:toolbox}-\ref{it:espscal} with $Q=P$ and noticing that $\bs$ is the Pettis integral of $(\log_{\bs})_{\#}P$ in the Hilbert space $\lin_{\bs}S$. 
\end{proof}

\begin{rem}
\label{rem:pkpos}
By definition of a barycenter, the right hand side of identity \eqref{eq:vi} is non-negative. It follows that 
\begin{equation}
    \label{eq:intkpos}
\int k^{b}_{\bs}(x)\,{\rm d}\P(x)\ge 0\,,
\end{equation}
for all $b\ne\bs\in S$. 

Note also that for any triple $x, b, \bs$,
the curvature properties of $S$ are directly reflected in $k^b_{\bs}(x)$.
Indeed, $k^b_{\bs}(x) \ge 1$ if $S$ is non-positively curved
and $k^b_{\bs}(x) \le 1$ if $S$ is positively curved according to~\eqref{eq:comparemetrics}.

\end{rem}

\subsection{A master theorem for spaces with curvature lower bound}
\label{subsec:lbc}

This section presents general statistical guarantees for empirical barycenters, in connection to lower bounds on the function $k_{\bs}^b$. 

Throughout, we fix a probability measure $P$ with barycenter $\bs$ on a geodesic space $(\Sp,d)$. Moreover, we denote by $\sigma^2$ the variance of $P$, which is defined by
\[\sigma^2=\int d^2(\bs,x)\,{\rm d}\P(x).\]

\begin{thm} 
 \label{thm:master} Suppose that $\curv(\Sp)\ge \kappa$, for some $\kappa\in \R$. If there exists a constant $k_{\rm min}>0$ such that $k_{\bs}^b(x)\ge k_{\min}$, for all $x,b\in \Sp$, then $\bs$ is unique and any empirical barycenter $b_n$ satisfies
 \[
 \esp d^2(b_n,\bs)\le \frac{4\sigma^2}{nk_{\rm min}^2}.
 \]

\end{thm}
\begin{proof}
Let
$k(x):=k_{\bs}^{b_n}(x)$ 
and recall the convention introduced in the proof of Theorem~\ref{thm:toolbox} concerning the use of the cone metric $\|\cdot-\cdot\|_{\bs}=|\cdot-\cdot|$ and associated bracket $\langle\cdot,\cdot\rangle_{\bs}=\langle\cdot,\cdot\rangle$. 
We also use the following standard notation:
\[
\P f(\bcdot)=\int f \ud \P\quad\mbox{and}\quad\P_nf(\bcdot)=\frac{1}{n}\sum_{i=1}^nf(X_i)\,.
\]

On the one hand, we have from the variance equality that 
\begin{align*}
  \P d^2(b_n, \bcdot)- \P d^2(\bs, \bcdot)&= d^2(b_n,\bs) \P k(\bcdot)\\
  \P_nd^2(\bs, \bcdot)- \P_nd^2(b_n, \bcdot)&= d^2(b_n,\bs) \P_n k_{b_n}^{\bs}(\bcdot).
\end{align*}
Adding these two identities yields
\begin{equation}
       \label{lem:unsurn:01} 
  (\P-\P_n)(d^2(b_n, \bcdot)-d^2(\bs, \bcdot))=     d^2(b_n,\bs)\left( \P k(\bcdot)+  \P_n k_{b_n}^{\bs}(\bcdot)\right)\,.
\end{equation}

On the other hand,  by the definition of $k$, we have
\begin{align*}
d^2(b_n, \bcdot)-d^2(\bs, \bcdot)&=(|b_n-\bcdot|^2-|\bs- \bcdot|^2)+(d^2(b_n, \bcdot)-|b_n- \bcdot|^2)\\
&= (|b_n- \bcdot|^2-|\bs- \bcdot|^2)+(k( \bcdot)-1)d^2(b_n,\bs).
\end{align*}
Hence, it follows that 
\begin{align}
(\P-\P_n)(d^2(b_n, \bcdot)-d^2(\bs, \bcdot))&=(\P-\P_n)(|b_n-\bcdot|^2-|\bs-\bcdot|^2) +d^2(b_n,\bs)(\P-\P_n)k(\bcdot).
\label{lem:unsurn:2}
\end{align}

Combining \eqref{lem:unsurn:01} and~\eqref{lem:unsurn:2} together yields
\begin{equation}
\label{lem:unsurn:3}
 d^2(b_n,\bs) \P_n [k(\bcdot)+k_{b_n}^{\bs}(\bcdot)]= (\P-\P_n)(|b_n-\bcdot|^2-|\bs-\bcdot|^2)\,.
\end{equation}

We now focus on the right-hand side of \eqref{lem:unsurn:3}. By definition of $\langle \cdot,\cdot\rangle$, it holds
\[|b_n-\bcdot|^2-|\bs-\bcdot|^2=|b_n|^2- 2\langle \bcdot,b_n\rangle.\]
Hence, by Theorem~\ref{thm:toolbox}-\ref{it:espscal} applied to $Q=P$ and $Q=P_n$, we get
\begin{align}
(\P-\P_n)(|b_n-\bcdot|^2-|\bs-\bcdot|^2)&=2\P_n \langle \bcdot,b_n\rangle
=2 \langle \bs_n,b_n\rangle\nonumber\\
&\le 2|\bs_n|\cdot|b_n|
= 2|\bs_n|d(b_n,\bs),
\label{lem:unsurn:4}
\end{align}
where $\bs_n$ stands for the barycenter of $(\log_{\bs})_{\#}(\P_n)$ in the Hilbert space $\lin_{\bs}\Sp\subset T_{\bs}\Sp$.
Combining \eqref{lem:unsurn:3} and \eqref{lem:unsurn:4}, yields
\begin{equation}
    \label{eq:main}
d(b_n,\bs)\P_n[k(\bcdot)+k_{b_n}^{\bs}(\bcdot)]\le 2|\bs_n|.
\end{equation}
Hence, since \(P_nk_{b_n}^{\bs}(\bcdot)\ge0\), if $k(x)\ge k_{\rm min}>0$ for all $x\in S$, we get by assumption that
\[
k^2_{\rm min}d^2(b_n,\bs)\le 4|\bs_n|^2.
\]

Since $(\log_{\bs})_{\#}(\P_n)$ is the empirical distribution associated to $(\log_{\bs})_{\#}(\P)$, we get by properties of averages in Hilbert spaces that
\begin{equation}
    \label{eq:main22}
    \esp|\bs_n|^2=\frac{\sigma^2}{n}\,,
\end{equation}
where we also used the fact that the variance of $(\log_{\bs})_{\#}(\P)$ is given by
\[
\int \|\log_{\bs}(x)\|_{\bs}^2 \ud \P(x)=\int d^2(x, \bs) \ud P(x)=\sigma^2\,.
\]
This completes the proof of the Theorem.

\end{proof}

As pointed out in Remark~\ref{rem:pkpos}, the condition $\curv(\Sp)\le 0$ implies that $k^{b_n}_{\bs}(x)\ge 1$ and $k^{\bs}_{b_n}(x)\ge 1$ for all $x\in S$.
The next result therefore follows readily from~\eqref{eq:main} and~\eqref{eq:main22} above.
\begin{cor} 
\label{cor:negcurv}
Suppose $(\Sp,d)$ is a geodesic space satisfying
$\kappa \le \curv(\Sp) \le 0$ for some $\kappa \le0$. Then, any distribution $\P\in\PS$ has a unique barycenter $\bs$ and any empirical barycenter $b_n$ satisfies
$$
\E d^2(b_n, \bs) \le \frac{\sigma^2}{n}.
$$
\end{cor}

A careful inspection of our proof techniques reveals that the previous result extends as well to complete and simply connected Riemannian manifolds with nonpositive sectional curvature, also
known as Hadamard manifolds. Indeed, the use of curvature lower bounds in our results is twofold.
First, it guarantees a Hilbert-like structure of the tangent cone at a barycenter.
Second, we concluded by using the fact that the pushforward of $P$ under $\log_{b^*}$ has the tip of the cone as its barycenter (i.e., the fact that any barycenter is a so-called \emph{exponential barycenter}; \cite{EmeMok91}).
The former is satisfied by the manifold assumption, and the latter is the content of Theorem 2.1(b,c) in \cite{BhaPat03}, at least when $M$ is Hadamard.

Our proof technique does not extend, however, to the case of general geodesic spaces with nonpositive curvature and no curvature lower bound. 
Indeed, it can be easily checked that claim~\ref{it:baryisexpbary} in Theorem~\ref{thm:toolbox} is typically not satisfied by distributions supported on the tripod (gluing of three line segments) for example.

\subsection{Tail bound}

We now provide an extension of Theorem \ref{thm:master} to the case of more general bounds on $k^b_{\bs}$. The next result shows in particular that the requirement, made in Theorem \ref{thm:master}, that $k^b_{\bs}(x)\ge k_{\rm min}>0$, for all $x,b\in S$, can be relaxed to $k^b_{\bs}(x)\ge k_{\rm min}$, for all $x,b\in S$ and $k_{\rm min}$ possibly negative, provided that $Pk_{\bs}(\bcdot)> 0$ where 
\[k_{\bs}(x)=\min_{b}k^b_{\bs}(x).\]

For simplicity, we prove this extension under an additional integrability assumption on $P$. We say that $P$ (with barycenter $\bs$) is \emph{subgaussian} with variance proxy $\varsigma^2>0$ if
\begin{equation}
\label{eq:subg}
    \int \exp\left(\frac{d^2(\bs,x)}{2\varsigma^2}\right)\ud P(x)\le 2.
    \end{equation}
Such conditions are commonly used in high-dimensional probability and statistics~\cite{Ver18, RigHut17}.

\begin{thm} 
 \label{thm:master2} Suppose that $\curv(\Sp)\ge \kappa$, for some $\kappa\in \R$. Fix a subgaussian probability distribution $P$ on $S$ with variance proxy $\varsigma^2>0$ and barycenter $\bs$.
 Suppose that there exists $k_{\rm min}<0$ such that $k^{b}_{\bs}(x)\ge k_{\rm min}$ for all $x,b\in\Sp$ and that  $Pk_{\bs}(\bcdot)> 0$. Then $\bs$ is unique and, for any empirical barycenter $b_n$ and any $\delta\in(0,1)$, we get
 \[
d^2(b_n,\bs)\le \frac{c_1}{n}\log\left(\frac{2}{\delta}\right),
 \]
 with probability at least $1-\delta-e^{-c_2n}$ where $c_1,c_2>0$ are independent of $n$.

 \end{thm}
\begin{proof} Denote $k(x)=k_{\bs}(x)$ for brevity.
For all $c\in(0,1)$ and $t>0$,
\begin{align}
    \prob(d(b_n,\bs)>t) &\le \prob(d(b_n,\bs)>t,P_nk\ge cPk)+\prob(P_nk<cPk)
    \nonumber\\
    &\le \prob(d(b_n,\bs)P_nk>tcPk)+\prob((P-P_n)k>(1-c)Pk)
    \nonumber\\
    &\le \prob(2|\bs_n|>tcPk)+\prob((P-P_n)k>(1-c)Pk)
    \nonumber,
\end{align}
where the last inequality follows from inequality~\eqref{eq:main}. Using a one-sided variant of Bernstein's inequality~\cite[eq. 2.10]{BouLugMas13}, we get 
\[\prob((P-P_n)k>(1-c)Pk)\le \exp\left(-\frac{n}{2}\frac{(1-c)Pk}{|k_{\rm min}|}\left(\frac{(1-c)|k_{\rm min}|Pk}{Pk^2}\land \frac{3}{2}\right)\right).\]
Using \eqref{eq:subg}, and the classical properties of subgaussian variables, we get
\begin{align}
    \prob(2|\bs_n|>tcPk)&\le 2\exp\left(-\frac{nt^2c^2(Pk)^2}{8\varsigma^2}\right).
\end{align}
The result then follows by considering, for any $c\in(0,1)$,
 \[c_1=\frac{c^2(Pk)^2}{8\varsigma^2}\quad\mbox{and}\quad c_2=\frac{(1-c)Pk}{2|k_{\rm min}|}\left(\frac{(1-c)|k_{\rm min}|Pk}{Pk^2}\land \frac{3}{2}\right).\]
\end{proof}

{To better appreciate the extent of this relaxation, recall the illustrative example of the sphere used in Section~\ref{sec:overview}. Theorem~\ref{thm:master2} covers distributions whose support may be almost the entire sphere. However, note that the existence of a finite $k_{\min}$ precludes a support that includes the cut-locus of the barycenter, even if $P$ puts arbitrarily small probability on small neighborhoods of the cut-locus. While such cases are out of the scope of the present paper, we anticipate that careful truncation arguments should suffice to derive parametric rates of convergence in such favorable cases.}

\section{Positively curved spaces}
Theorem~\ref{thm:master} guarantees a dimension-free parametric rate of convergence of $b_n$ under a uniform positive lower bound on the function $(x,b)\mapsto k^b_{\bs}(x)$. Such a condition is closely linked to the better know notion of $k$-convexity which is connected to positive curvature upper bounds~\cite{Oht07}. In fact, it is not hard to check that  assuming that $k^b_{\bs}(x)\ge k_{\min}$ uniformly not only in $b,x$ but also in $\bs \in \supp(P)$ is equivalent to  $k_{\min}$-convexity of the support of $P$. As a result, the condition of  Theorem~\ref{thm:master} is weaker than $k_{\min}$-convexity of the support of $P$ since it imposes only a control on geodesics emanating from $\bs$. The rest of this section is devoted to developing a suitable relaxation via the notion of geodesic extendibility.

\subsection{Extendible geodesics}

We present a compelling synthetic geometric condition that implies this lower bound in the context of positively curved spaces: the extendibility, by a given factor, of all geodesics (emanating from and arriving at) the barycenter $\bs$.
To formalize the notion of extendible geodesics, consider a (constant-speed) geodesic $\gamma \colon [0,1]\to \Sp$.
For $(\lambda_{\rm in},\lambda_{\rm out})\in \R^2_+$, we say that $\gamma$ is $(\lambda_{\rm in},\lambda_{\rm out})$-extendible if there exists a (constant-speed) geodesic $\gamma^{+} \colon [-\lambda_{\rm in},1+\lambda_{\rm out}]\to \Sp$ such that $\gamma$ is the restriction of $\gamma^{+}$ to $[0,1]$.
Before we state our sufficient condition we shall first need a fact
given in
\cite{AhiGouPar18}. 
\begin{thm}
\label{thm:extendvariance}
Suppose that $\curv(\Sp) \ge 0$. Let $\P\in\PS$ with barycenter $\bs$.
Suppose that, for each $x \in \supp(\P)$, there exists a geodesic $\gamma_{x}:[0,1]\to\Sp$ connecting $\bs$ to $x$ which is $(0,\lambda)$-extendible. Suppose in addition that $\bs$ remains a barycenter of distribution $\P_{\lambda}=(e_{\lambda})_{\#} P$ where $e_{\lambda}(x)=\gamma^{+}_{x}(1+\lambda)$. Then for all $b \in \Sp$,
\begin{equation}
\label{eq:extend}    
\frac{\lambda}{1 + \lambda} d^2(b, \bs) \le  \int (d^2(x, b) - d^2(x,\bs)) \ud \P(x).
\end{equation}
\end{thm}
According to the variance equality of Theorem~\ref{thm:varequal}, inequality \eqref{eq:extend} is equivalent to the statement
that, for all $b \in \Sp$,
\[
\int k_{\bs}^b(x)\ud\P(x) \ge \frac{\lambda}{1 + \lambda}.
\]
We will use this observation next.

\begin{thm}
\label{thm:extendgeod} Suppose that $\curv(\Sp)\ge 0$ and let $x, b, \bs \in \Sp$. Suppose that, for some $\lambda_{\rm in},\lambda_{\rm out}>0$, there is a geodesic connecting $\bs$ to $x$ which is $(\lambda_{\rm in},\lambda_{\rm out})$-extendible. Then 
\[
k_{\bs}^b(x)\ge \frac{\lambda_{\rm out}}{1+\lambda_{\rm out}}-\frac{1}{\lambda_{\rm in}}.
\]
\end{thm}

\begin{proof} 
Let $\gamma:[0,1]\to \Sp$ be a $(\lambda_{\rm in},\lambda_{\rm out})$-extendible geodesic connecting $\bs$ to $x$ and Let $\gamma^+:[-\lambda_{\rm in},1+\lambda_{\rm out}]\to S$ be its extension. Let $z=\gamma^+(-\xi)$ where $\xi=\lambda_{\rm in}/(1+\lambda_{\rm out})$.
Then, it may be easily checked that $\bs$ is a barycenter of the probability measure 
\[
\P:=\frac{\xi}{1+\xi}\delta_x+\frac{1}{1+\xi}\delta_{z}.
\] Now, we wish to apply Theorem~\ref{thm:extendvariance} to $\P$. To this aim, note that the geodesic $\gamma$ from $\bs$ to $x$ is $(0,1+\lambda_{\rm out})$-extendible by assumption with $e_{\lambda_{\rm out}}(x)=\gamma^+(1+\lambda_{\rm out})$. Similarly, we check that the geodesic $\sigma:[0,1]\to S$ connecting $\bs$ to $z$ and defined by $\sigma(t)=\gamma^+(-t\xi)$ is $(0,1+\lambda_{\rm out})$-extendible by construction with $e_{\lambda_{\rm out}}(z)=\gamma^+(-\lambda_{\rm in})$. Finally, one checks that $\bs$ remains a barycenter of the probability measure $P_{\lambda_{\rm out}}=(e_{\lambda_{\rm out}})_{\#}P$. As a result, Theorem~\ref{thm:extendvariance} implies that
\[    
    \frac{\lambda_{\rm out}}{1 + \lambda_{\rm out}}\le Pk_{\bs}^b(\bcdot)= \frac{\xi}{1 + \xi} k_{\bs}^b(x) + \frac{1}{1 + \xi}k_{\bs}^b(z).
    \]
Finally, the fact that $\curv(S)\ge 0$ implies that $d(x,y)\le \|\log_{\bs}(x)-\log_{\bs}(y)\|_{\bs}$, for all $x,y\in S$ which imposes that $k^b_{\bs}(z)\le 1$ for all $b,z\in S$. Hence, we obtain
\begin{align*}
    k_{\bs}^b(x) &\ge \frac{1 + \xi}{\xi} \left( \frac{\lambda_{\rm out}}{1 + \lambda_{\rm out}} - \frac{1}{1 + \xi}\right)= \frac{\lambda_{\rm out}}{1 + \lambda_{\rm out}} -
    \frac{1}{\lambda_{\rm in}},
\end{align*}  
which completes the proof.
\end{proof}

\subsection{The Wasserstein space}\label{subsec:wasserstein}

The case of geodesic spaces with lower curvature bound contains several spaces of interest such as the space of metric measure spaces equipped with the Gromov-Wasserstein distance~\cite{Stu12} or the Wasserstein space over a postively curved space.
In this subsection, we explore the extendible geodesics condition, described in the previous subsection, in the context of the Wasserstein space over a {separable} Hilbert space. 
This space exhibits a very particular structure that allows for a simple and transparent formulation of the extendibility of geodesics in terms of the regularity of Kantorovich potentials.
Given a {separable} Hilbert space $H$, with inner product $\langle\cdot ,\cdot\rangle$ and norm $|\cdot|$, recall that the subdifferential $\partial \varphi\subset H^2$ of a function $\varphi:H\to \R$ is defined by $\partial \varphi=\{(x,g):\forall y\in H, \varphi(y)\ge\phi(x)+\langle g,y-x\rangle\}$.
We denote $\partial \varphi(x)=\{g\in H:(x,g)\in \partial\varphi\}$.
The function $\varphi$ is called $\alpha$-strongly convex if, for all $x\in H$, $\partial \varphi(x)\ne\emptyset$ and 
\[
\langle g,x-y\rangle\ge\phi(x)-\varphi(y)+\frac{\alpha}{2}|x-y|^2,
\]
for all $g\in\partial\varphi(x)$ and all $y\in H$.
We recall the following result.
\begin{thm}[Theorem 3.5 in \cite{AhiGouPar18}]
\label{thm:halfextwass}
Let $(\Sp,d)=(\mathcal P_2(H),W_2)$ be the Wasserstein space over a separable Hilbert space $H$.
Let $\mu$ and $\nu$ be two elements of $\Sp$ and let $\gamma:[0,1]\to S$ be a geodesic connecting $\mu$ to $\nu$.
Then, $\gamma$ is $(0,\lambda)$-extendible if, and only if, the support of the optimal transport plan of $\mu$ to $\nu$ lies in the subdifferential $\partial\phi_{\mu\ra\nu}$ of a $\frac{\lambda}{1+\lambda}$-strongly convex map $\phi_{\mu\ra\nu}:H\to \R$.
\end{thm}

In the above theorem, $\phi_{\mu\ra\nu}$ is defined as follows. Let $f_{\mu\to\nu},g_{\nu\to\mu}$ be optimal Kantorovich potentials, i.e., solutions of the dual Kantorovich problem
\[W^2_2(\mu,\nu)=\sup\left\{\int f\ud \mu+\int g\ud\nu: f,g\in C_b(H), f(x)+g(y)\le |x-y|^2\right\}\,,\]
where $C_b(H)$ denotes the set of real-valued bounded continuous functions on $H$. 
Then, for all $x\in H$,
\[f_{\mu\to\nu}(x)=2|x|^2-\phi_{\mu\ra\nu}(x),\qquad g_{\nu\to\mu}(x)=2|x|^2-\phi_{\nu\ra\mu}(x),\] 
and one checks that $\phi_{\nu\ra\mu}=\phi^{*}_{\mu\ra\nu}$ is the Fenchel-Legendre conjuguate of $\phi_{\mu\ra\nu}$ ~\cite[Theorem 5.10]{Vil09}. Our next result is in the same spirit as Theorem~\ref{thm:halfextwass} and characterises the $(\lambda_{\rm in}, 1+\lambda_{\rm out})$-extendibility of geodesics in the Wasserstein space in a specific scenario. Recall that a convex function $\varphi:H\to \R$ is called $\beta$-smooth if

\[
\langle g_x,x-y\rangle\le\phi(x)-\varphi(y)+\frac{\beta}{2}|x-y|^2,\qquad \forall\, g_x \in \partial \varphi(x)
\]
for all $x,y\in H$. In fact, for $\beta>0$, it is easy to check that convex $\beta$-smooth functions are differentiable. Therefore, a convex function $\varphi:H\to \R$ is  $\beta$-smooth if and only if it is differentiable and 
\[
\langle \nabla \varphi(x),x-y\rangle\le\phi(x)-\varphi(y)+\frac{\beta}{2}|x-y|^2\,.
\]

It is known that a convex function is $\beta$-smooth if and only if its Fenchel-Legendre transform is $1/\beta$-strongly convex \cite[Theorem 18.15]{BauCom17}. 
Since $\phi_{\nu\ra\mu}=\phi^{*}_{\mu\ra\nu}$, the next result follows readily from Theorems~\ref{thm:master},~\ref{thm:extendgeod} and~\ref{thm:halfextwass}.

\begin{cor}
\label{cor:wasserstein}

Let $(\Sp,d)=(\mathcal P_2(H),W_2)$ be the Wasserstein space over a separable Hilbert space $H$. Let $\P\in\PS$ with barycenter $\mus\in\Sp$. Let $\alpha,\beta>0$ and suppose that every $\mu\in\supp(P)$ is the pushforward of $\mu^{\star}$ by the gradient of an $\alpha$-strongly convex and $\beta$-smooth function $\varphi_{\mus\to\mu}$, i.e., $\mu=(\nabla\varphi_{\mus\to\mu})_{\#}\mus$. Then, if $\beta-\alpha<1$, any empirical barycenter $\mu_n$ of $P$ satisfies
\[
\esp W_2^2(\mu_n,\mus)\le  \frac{4\sigma^2}{(1-\beta+\alpha)^2n}.
\]
\end{cor}

\subsection{Examples}
We complete this section by two concrete examples in which we instantiate sufficient conditions for the application of Corollary~\ref{cor:wasserstein}.

\subsubsection{Gaussians.}

{Consider first the metric space of Gaussians distributions over $H=\R^D$ equipped with the $2$-Wasserstein distance.

Assume that $P$ is supported on non-degenerate
Gaussians $\mathcal{N}(m, \Sigma) \sim P$.

It is a well-known fact that the optimal transport map between
$\mathcal{N}(m_0, \Sigma_0)$ and $\mathcal{N}(m_1, \Sigma_1)$ is given by~\cite{PeyCut19}:
\begin{equation}\label{eqn:bw_map}
x \mapsto \Sigma_0^{-1/2}(\Sigma_0^{1/2} \Sigma_1 \Sigma_0^{1/2})^{1/2} \Sigma_0^{-1/2}(x - m_0) + m_1.
\end{equation}

Using Equation~\eqref{eqn:bw_map}, the condition of
geodesic extendibility
translates, by Theorem~\ref{thm:halfextwass},
to control on the maximum and minimum eigenvalues of
the matrix acting on $(x - m_0)$.
A natural means of establishing this control uniformly
for the distribution $P$ is to assume that each element of its support
$\mathcal{N}(m, \Sigma)$ has covariance matrix $\Sigma$ with eigenvalues
in $[\kappa_0, \kappa_1]$.
In terms of the space of positive definite matrices,
this can be interpreted as ensuring that the support of $P$ lies away from the boundary while remaining
bounded, respectively.

Under this assumption on the support of $P$, it can be shown~\cite[Prop. 15]{chewi2020gradient}
that the
barycenter of $P$ exists, and is the unique Gaussian
$\mathcal{N}(m_{\star}, \Sigma_{\star})$
where $m_{\star} := \int m \ud P(m)$
and $\Sigma_{\star}$ solves the first-order optimality equation
\[
I_D = \int \Sigma_{\star}^{-1/2}(\Sigma_{\star}^{1/2} \Sigma
\Sigma_{\star}^{1/2})^{1/2}
\Sigma_{\star}^{-1/2}
\ud P(\Sigma).
\]
Using this equation one can directly calculate that
$\Sigma_{\star}$ must also have eigenvalues in
$[\kappa_0, \kappa_1]$.
By inspecting the formula~\eqref{eqn:bw_map}, it follows that the optimal transport maps between $\mathcal{N}(m_\star,\Sigma_{\star})$ 
and each element of the support of $P$ are
$\kappa^{-1}$-strongly convex and $\kappa$-smooth, where $\kappa:= \kappa_1/\kappa_0$.
Together, this yields the following application of our main result.
\begin{cor} Let $(S,d)$ denote the space of Gaussians metrized by the $2$-Wasserstein distance. Fix $\kappa_1>\kappa_0>0$ and let $P \in \mathcal{P}_2(S, d)$ be any distribution supported solely
on Gaussians
$\mathcal{N}(m, \Sigma)$ with
eigenvalues in $[\kappa_0, \kappa_1]$. Denote $\kappa := \kappa_1/\kappa_0$,
and assume $\kappa - \kappa^{-1} < 1$. 
Then $P$ has a unique barycenter $\mu^{\star} =
\mathcal{N}(m_{\star}, \Sigma_{\star})$, and the empirical barycenter
$\mu_n = \mathcal{N}(m_n, \Sigma_n)$  satisfies
\[
\E W_2^2(\mu_n, \mu^{\star}) \leq \frac{4 \sigma^2}{(1 - \kappa + \kappa^{-1})^2n}.
\]
\end{cor}}

\subsubsection{Template deformation model.}
Throughout this paper we have made minimal assumptions on $P$, apart from those that ensure convergence of its barycenter. It is customary to define models where the mean, and more generally the barycenter of $P$ is a parameter of interest. Template deformation models were introduced as convenient models where the Wasserstein barycenter is the parameter of interest. As we will see below, this model gives an equivalent but complementary perspective on Corollary~\ref{cor:wasserstein}.

Recall that since $(\mathcal{P}_2(\R^D),W_2)$ is positively curved, the barycenter $\mu^\star$ of a measure $P\in\mathcal{P}_2(\mathcal{P}_2(\R^D))$ is also a barycenter of the measure $P$ pushed forward to the tangent space at $\mu^\star$; in other words, the barycenter is also an \emph{exponential} barycenter~ \cite[Corollary~2]{Le-19}.
Moreover, the tangent space at a measure $\mu^\star$ can be identified with the closure in $L^2(\mu^\star)$ of the set of maps $T$ that can be written as the gradient of a function --- this follows from identifying a tangent vector $\log_{\mu^\star}(\mu)$ with $\nabla\phi_{\mus\to\mu}-\id$ where $\nabla\phi_{\mus\to\mu}$ is the optimal transport map from $\mu^\star$ to $\mu$; see \cite[\S 8.4]{AmbGigSav05}.
In particular, $\mus$ is an exponential barycenter for $P$ if and only if
\begin{equation}\label{eq:expbary}
\int \nabla\phi_{\mus\to\mu}(x)\ud P(\mu)=x,\, \text{ for $
\mu^\star$-almost all }x.
\end{equation}

Exponential barycenters correspond to critical points that are not necessarily global minima of the variance functional. Hence, \eqref{eq:expbary} does not imply that  $\mu^\star$ is a barycenter of $P$.
However, one can easily check that the variance equality (Theorem~\ref{thm:varequal}) holds also when $\mus$ is an exponential barycenter.
It implies that if \eqref{eq:expbary} holds and $k_{\mus}^\nu(\mu)\ge 0$ for all $\nu\in\mathcal{P}_2(\R^D)$ and $P$-almost all $\mu$, then $\mu^\star$ is also a barycenter of $P$.

Wasserstein exponential barycenters as in~\eqref{eq:expbary} arise naturally as the parameter of interest in the following template deformation model \cite{boissard2015distribution, ZemPan19}. Let $T\in L^2(\R^D,\R^D,\mus)$ be a random function, often call \emph{warping function}. In a template deformation model, one observes independent copies of 
\begin{equation}
    \label{eq:templatedef}
    \mu=T_{\#}\mus\,, \quad T \sim Q
\end{equation}
and the goal is to estimate $\mus$.
Under some conditions on the distribution $Q$ of $T$, $\mus$ is an exponential barycenter of the distribution of $\mu$. For example if $T$ is linear and identified to a positive semidefinite $D \times D$ matrix, it is sufficient to assume the $\E_Q[T]=I_D$. More generally, we can assume that $T$ is a gradient of a convex function. Since $T_\#\mus=\mu$, we have by Brenier's theorem that $T=\nabla \varphi_{\mus\to \mu}$. In that case, $\E_Q[T]=I_D$ is replaced with~\eqref{eq:expbary}. Additional regularity conditions on the warping functions ensures that $\mus$ is, in fact, a barycenter.

The following result follows readily from Corollary~\ref{cor:wasserstein}.
\begin{cor}
Fix $\alpha, \beta>0$ such that $\beta - \alpha < 1$ and consider the template deformation model~\eqref{eq:templatedef} where $Q$ is supported on gradients of $\alpha$-strongly convex and $\beta$-smooth function $\phi_{\mus\to\mu}$.  Moreover, assume that
$$
\int T(x)\ud Q(T)=x,\, \text{ for $
\mu^\star$-almost all }x.
$$
Let $\mu_n$ denote the empirical barycenter of $n$ independent copies of $\mu$ from the template deformation model~\eqref{eq:templatedef}.
Then, for all $n\ge1$, it holds
\[
\esp W_2^2(\mu_n,\mus)\le  \frac{4\sigma^2}{(1-\beta+\alpha)^2n}\,,
\]
where 
\[
\sigma^2=\E W_2^2(\mu, \mus)=\int \|T(x)-x\|^2 \ud \mus (x)\ud Q(T).
\]
\end{cor}

\section*{Acknowledgments}
TLG, PR and QP acknowledge the hospitality of the Institute for Advanced Study where most of this work was carried out. TLG and QP were supported by the Russian Academic Excellence Project 5-100. PR was supported by NSF awards IIS-1838071, DMS-1712596 and DMS-TRIPODS-1740751; grant 2018-182642 from the Chan-Zuckerberg Initiative DAF and the  Schmidt Foundation.
The authors thank Sinho Chewi and anonymous reviewers for helpful remarks.

\bibliographystyle{aomalpha}
\bibliography{extracted}

\end{document}